\documentclass[10pt,a4paper]{article}

\author{Robson Ricardo de Araujo \qquad Grasiele C. Jorge}
\title{Construction of full diversity $D_n$-lattices for all $n$}

\usepackage{geometry}
\usepackage[latin1]{inputenc}
\usepackage[T1]{fontenc}
\usepackage{syntonly}
\usepackage{verbatim}
\usepackage{makeidx}
\usepackage{latexsym}
\usepackage{graphicx,color}
\usepackage{graphicx}
\usepackage{amsthm,amsfonts}
\usepackage{ulem}
\usepackage{amssymb}
\usepackage{amsmath}
\usepackage{hyperref}
\usepackage{lipsum}
\usepackage[all,cmtip]{xy}
\usepackage{mathtools}

\newtheorem{teorema}{Theorem}[section]
\newtheorem{ex}{Example}[section]
\newtheorem{corol}{Corollary}[section]

\newtheorem{defin}{Definition}[section]
\newtheorem{lema}{Lemma}[section]
\newtheorem{obs}{Remark}[section]

\def\linha#1{%
  \hbox to \hsize{%
      \vbox{\centering #1}}%
      \vspace{4mm}}

\begin{document}

 \maketitle

\section{Introduction}

Lattice coding have been used mainly for Gaussian channels (\cite{sloane}). For these channels it is important that the lattices have good sphere packing, that is, have high density. Lattices have been used also to obtain codes for Rayleigh fading channel. In this case, it is desired to get lattices having maximum diversity and great minimum product distance. Thus, in view of possible applications to these two types of channels it is desirable to find lattices having maximum diversity and having, at the same time, high density and great minimum product distance (\cite{boutros}).

For general lattices it is not a simple task to find lattices having maximum diversity and to estimate their minimum product distance. For lattices obtained via number fields through a twisted homomorphism (\textit{algebraic lattices}), this task can be easier. Totally real number fields produce algebraic lattices having maximum diversity and there is a closed form for the minimum product distance when principal ideals inside the ring of integers of these number fields are considered.

In \cite{bayer-oggier} rotated versions of the lattice $\mathbb{Z}^n$ are constructed via totally real number fields when $n$ is a prime number or $n=(p-1)/2$ for some prime number $p$ and also for others mixed values of $n$. In that paper the authors discuss the minimum product distance of those lattices in several examples. More generally, it is possible to construct rotated $\mathbb{Z}^n$-lattices for an odd number $n$ (\cite{sethoggier}). Full diversity constructions of $\mathbb{Z}^n$ are also known for $n$ a power of $2$ (\cite{bayer-nebe}, \cite{andrade}). As it is well known, the center density of the lattice $\mathbb{Z}^n$ gets very low as $n$ grows. Yet, the sublattices $D_n$ have higher density. For $n=3,4,5$ it is known that $D_3$, $D_4$ and $D_5$ have the higher possible packing density in these dimensions (\cite{sloane}). In \cite{grasi2} it is presented the construction of rotated lattices $D_n$ and their product distance for $n$ equal to a power of $2$ and for $n=(p-1)/2$, where $p\geq 3$ is a prime number (see also \cite{grasi}).

Following this direction, in this work we use the algebraic construction of $\mathbb{Z}^n$ of \cite{sethoggier} (for $n=2^m>1$ and $n>1$ an odd number) to obtain its sublattice $D_n$ with maximum diversity. We then consider mixed constructions to obtain rotated versions of $D_n$ for any $n\geq 2$ with maximum diversity (Theorem \ref{teorema9}). To obtain a closed form for the product distance of the constructed rotated lattices we prove first that the $\mathbb{Z}$-module used to construct $\mathbb{Z}^n$, $n$ odd, is an ideal and provide a sufficient condition to this ideal to be principal (Theorems \ref{teorema14} e \ref{teorema2}). Assuming this condition, the minimum product distance of $\mathbb{Z}_n$ is obtained (Corollary \ref{corol1}) and a bound for this distance in $D_n$ is derived (Corollary \ref{corol3}). Under the same above condition, we extend these results for the mixed constructions (Theorems \ref{teorema15} and \ref{teorema11}). In particular, for example, we see that the rotated lattice $D_6$ obtained here has better minimum product distance than the obtained in \cite{grasi} via other construction.

This work is organized as follows. In Section \ref{sec0} some preliminary concepts and results of algebraic lattices are introduced. In Section \ref{sec1} a rotated version of $D_n$, for any odd number $n$, is presented. In Section \ref{sec2} we prove that the $\mathbb{Z}$-module used to obtain $\mathbb{Z}^n$ is an ideal and analyse the minimum product distance of $\mathbb{Z}^n$ and $D_n$, $n$ odd, given a condition. In Section \ref{sec4}, we obtain $D_n$ from a known construction of $\mathbb{Z}^n$, for $n$ a power of two, and present a mixed construction and study the minimum product distance of $\mathbb{Z}^n$ and $D_n$, for $n$ an even integer number. Finally, in Section \ref{sec5} we discuss the advantage of using $D_n$ instead of $\mathbb{Z}^n$ looking some examples.

\section{Preliminaries about lattices}\label{sec0}

A lattice of rank $k\leq n$ is a discrete additive subgroup of $\mathbb{R}^n$. Equivalently, a lattice of rank $k$ is a set generated by $k$ linearly independent elements of $\mathbb{R}^n$ over $\mathbb{Z}$. If $k=n$ we say the lattice is complete. All lattices treated in this work are complete. Because of this, from now on this information will be omitted.

A set of generators of a lattice $\Lambda\subset \mathbb{R}^n$ is called basis. Considering $\{v_1,\ldots,v_n\}$ a basis of $\Lambda$, the matrix $n\times n$ whose each $i$-th line is formed by the entries of the vector $v_i$ is called generator matrix of $\Lambda$. An element $x\in\mathbb{R}^n$ belongs to a lattice $\Lambda$ having generator matrix $M$ if and only if there exists $y\in\mathbb{Z}^n$ such that $x=yM$. The square matrix $G=MM^T$ is called Gram matrix of $\Lambda$, where $M^T$ denotes the transpose matrix of $M$. Two different basis of $\Lambda$ give different generator matrices, but they give the same Gram matrix. So, the determinant $det(\Lambda)$ of a lattice $\Lambda$ is defined to be the determinant of its Gram matrix. The volume of a lattice $\Lambda$ is defined to be $vol(\Lambda):=\sqrt{|det(\Lambda)|}$ and coincides with the volume of the fundamental region of $\Lambda$, that is defined by
$$\mathcal{P}(B)=\{a_1v_1+\ldots+a_nv_n~:~a_i\in[0,1), i=1,\ldots,n\}$$
where $B=\{v_1,\ldots,v_n\}$ is a basis of $\Lambda$. The norm of the nonzero vector in $\Lambda$ having the lowest norm among all elements in $\Lambda$ is called minimum norm of $\Lambda$ and is denoted by $\lambda$. Any additive subgroup $\Lambda'$ of $\Lambda$ is called sublattice of $\Lambda$.

This work talks about two important families of lattices: $\mathbb{Z}^n$ and $D_n$. The lattice $\mathbb{Z}^n\subset\mathbb{R}^n$ is called cubic lattice and is generated by $e_i=(0,\ldots,0,1(i-th),0,\ldots,0)$, $1\leq i\leq n$. $\mathbb{Z}^n$ has determinant $1$, volume $1$ and minimum norm $1$. In turn, $D_n$ is a sublattice of $\mathbb{Z}^n$ given by
$$D_n=\{(x_1,\ldots,x_n)\in\mathbb{Z}^n~:\exists m\in\mathbb{Z}~s.t.~x_1+x_2+\ldots+x_n=2m\}$$
that has determinant $4$, volume $2$ and minimum norm $\sqrt{2}$.

The center density of a lattice $\Lambda\subset\mathbb{R}^n$ is defined to be $\delta = \rho^n/vol(\Lambda)$, where $\rho=\lambda/2$ is the packing radius of $\Lambda$. Find a lattice having high center density in a certain dimension is a task related to the problem of the sphere packing (see \cite{sloane}). Specially, lattices having good center density are useful for gaussian channels in the Coding Theory.

We say a lattice $\Lambda\subset \mathbb{R}^n$ has maximum diversity if for all $x=(x_1,\ldots,x_n)\in\Lambda$ such that $x\neq 0$ then $x_i\neq 0$ for any $i=1,\ldots,n$. If $\Lambda$ has maximum diversity, we define the minimum product distance of $\Lambda$ by
$$d_{p,min}(\Lambda)=\inf\{|x_1\ldots x_n|~:~0\neq x=(x_1,\ldots,x_n)\in\Lambda\}.$$
Lattices having maximum diversity and great minimum product distance are useful for Rayleigh fading channels.

If a lattice $\Lambda_1$ is obtained from an other $\Lambda_2$ by a rotation or by a scalling, we say that $\Lambda_1$ and $\Lambda_2$ are equivalent. Particularly, if a lattice $\Lambda_1$ is obtained from an other $\Lambda_2$ by a rotation we can say that $\Lambda_1$ is a rotated version of $\Lambda_2$. Equivalent lattices have same center density, but one of them can have maximum diversity and the other not. Also, minimum product distances of equivalent lattices having maximum diversity can be different. 

It is possible obtain lattices from the Number Theory. Consider $\mathbb{K}$ a totally real number field of degree $n$, $\mathcal{O}_\mathbb{K}$ the ring of integers of $\mathbb{K}$ and $M\subset\mathcal{O}_\mathbb{K}$ a $\mathbb{Z}$-module of rank $n$. Suppose that $\sigma_1,\ldots,\sigma_n$ are the monomorphisms from $\mathbb{K}$ to $\mathbb{R}$. Let $\beta\in\mathcal{O}_\mathbb{K}$ be a number such that $\beta_i:=\sigma_i(\beta)>0$, for $i=1,\ldots,n$. For this, we say that $\beta$ is a totally positive number in $\mathbb{K}$. We define the twisted homomorphism $\sigma : \mathbb{K}\longrightarrow \mathbb{R}^n$ by
$$\varphi_\beta(x)=\left(\sqrt{\beta_1}\sigma_1(x),\ldots,\sqrt{\beta_n}\sigma_n(x)\right)$$
for any $x\in \mathbb{K}$. This definition can be generalized for any number field $\mathbb{K}$ (see \cite{boutros}). So, the set $\varphi_\beta(M)$ is a lattice of rank $n$ in $\mathbb{R}^n$ called algebraic lattice. It has volume equal to $\sqrt{|d_\mathbb{K}N_\mathbb{K}(\beta)|}N(M)$, where $d_\mathbb{K}$ denotes the discriminant of the field $\mathbb{K}$, $N_\mathbb{K}(\beta)$ denotes the algebraic norm of $\beta$ in the extension $\mathbb{K}/\mathbb{Q}$ and $N(M)$ denotes the index $[\mathcal{O}_\mathbb{K}:M]$. Besides that, $\varphi_\beta(M)$ has maximum diversity and minimum distance product given by $d_{p,min}(\varphi_\beta(M))=\sqrt{N_\mathbb{K}(\beta)}\min_{0\neq x\in M} |N_\mathbb{K}(x)|$. When $M$ is a principal ideal the lattice $\varphi_\beta(M)$ has minimum product distance given by $d_{p,min}(\varphi_\beta(M))=vol(\varphi_\beta(M))/\sqrt{d_\mathbb{K}}$. More about algebraic lattices can be seen in \cite{grasi2}.

The purpose of this work is obtain $D_n$ for any $n>1$ via totally real number fields. In second analysis, we want to calculate a more explicit expression for the minimum product distance of $D_n$ made by that construction when this is possible.

\section{Algebraic lattices $\mathbb{Z}^n$ and $D_n$, for any odd number $n>1$} \label{sec1}

Consider $n$ a odd number bigger than $1$. Due to Dirichlet Theorem (\cite{serre}, Chapter 3, Lemma 3), there exists a prime number $p$ such that $p\equiv 1~(mod~n)$. Denote the $p$-th primitive root of unity $e^{\frac{2i\pi}{p}}$ by $\zeta_p$. The cyclotomic extension $\mathbb{Q}(\zeta_p)/\mathbb{Q}$ has cyclic Galois group generated by $\sigma$ defined by $\sigma(\zeta_p)=\zeta_p^r$, in which $r$ is a primitive element of the field $(\mathbb{Z}/p\mathbb{Z})^*$ (that is, $r$ is an element such that its lower power $j>0$ satisfying $r^j\equiv 1~(mod~p)$ is $p-1$).

The subgroup $H=\langle \sigma^n \rangle$ of $Gal(\mathbb{Q}(\zeta_p)/\mathbb{Q})$ has a subfield of $\mathbb{Q}(\zeta_p)$ as fixed field, which we denote by $\mathbb{K}$, that is,
\begin{equation}\label{eq4}
\mathbb{K}=\{ y\in\mathbb{Q}(\zeta_p)~:~\sigma^n(y)=y\}.
\end{equation}
The degree of $\mathbb{K}$ is $n$. Besides, $\mathbb{K}$ is contained in the maximal real subfield $\mathbb{Q}(\zeta_p+\zeta_p^{-1})$. Then, $\mathbb{K}$ is a totally real number field.

Consider in $\mathbb{Q}(\zeta_p)$ the element
$$\alpha = \prod_{j=0}^{m-1} \left(1-\zeta_p^{r^j}\right)$$
in which $m=(p-1)/2$. Since $p$ is prime and $r<p$ then $mdc(r-1,p)=1$ and, consequently, there exists an integer $\lambda$ satisfying $\lambda(r-1)\equiv 1~(mod~p)$. Now, consider also in $\mathbb{Q}(\zeta_p)$ the element
$$z=\zeta_p^\lambda\alpha(1-\zeta_p).$$
Note that $z$ is an algebraic integer. Because of this, the element
$$x=Tr_{\mathbb{Q}(\zeta_p):\mathbb{K}}(z)=\sum_{j=1}^\frac{p-1}{k} \sigma^{jk}(z)$$
is an element belonging to $\mathcal{O}_\mathbb{K}$.

\begin{lema}[\cite{bayer-oggier}, lemams 3 and 4] The following equalities are true: \label{lema1}
\begin{enumerate}
	\item[a)] $\sigma(\alpha)=-\zeta_p^{p-1}\alpha$
	\item[b)] $\sigma(\zeta_p^\lambda \alpha) = -\zeta_p^\lambda \alpha$
	\item[c)] $(\zeta_p^\lambda \alpha)^2 = (-1)^m p$
\end{enumerate}
\end{lema}

\begin{lema}[\cite{elia}, appendix II]\label{lema2} $Tr_\mathbb{K}(x^2)=p^2$ and $Tr_\mathbb{K}(x\sigma^j(x))=0$ if $j\neq 0$.
\end{lema}

\begin{teorema} It is ortoghonal the matrix \label{teorema6}
	$$G=\frac{1}{p}\left(\begin{array}{ccccc}
	x & \sigma(x) & \ldots & \sigma^{n-2}(x) & \sigma^{n-1}(x)\\
	\sigma(x) & \sigma^2(x) & \ldots & \sigma^{n-1}(x) & x\\
	\sigma^2(x) & \sigma^3(x) & \ldots & x & \sigma(x)\\
	\vdots & \vdots & \vdots & \vdots & \vdots\\
	\sigma^{n-1}(x) & x & \ldots & \sigma^{n-3}(x) & \sigma^{n-2}(x)
	\end{array}\right).$$
that is, $GG^T=G^TG=I_n$
\end{teorema}
\begin{proof}
This follows directly from the Lemma \ref{lema2}.
\end{proof}

The above proposition allows us to construct the algebraic lattice $\mathbb{Z}^n$ through the twisted homomorphism with $\beta=1/p^2$ and with the $\mathbb{Z}$-module
\begin{equation} \label{eq3}
I=\langle x,\sigma(x),\ldots, \sigma^{n-1}(x)\rangle_\mathbb{Z}.
\end{equation}

In turn, the following theorem presents the construction of the algebraic lattice $D_n$ through a $\mathbb{Z}$-module inside the field $\mathbb{K}$.

\begin{teorema} Consider $\beta=1/p^2$ and $M$ the $\mathbb{Z}$-module generated by \label{teorema1}
$$\{x+\sigma(x),x-\sigma(x),\sigma(x)-\sigma^2(x),\ldots,\sigma^{n-2}(x)-\sigma^{n-1}(x)\}.$$
Thus, the algebraic lattice $\sigma_\beta(M)$ is a rotated version of the lattice $D_n$.
\end{teorema}
\begin{proof} A generator matrix of $D_n$ is given by
\begin{equation}\label{eq6}
\left(\begin{matrix}
-1 & -1 & 0 & \ldots & 0 & 0\\
1 & -1 & 0 & \ldots & 0 & 0\\
0 & 1 & -1 & \ldots & 0 & 0\\
\vdots & \vdots & \vdots & \ddots & \vdots & \vdots\\
0 & 0 & 0 & \ldots & 1 & -1
\end{matrix}\right).
\end{equation}
Multiplying the above matrix by the ortoghonal matrix $G$, we obtain
$$\frac{1}{p}\left(\begin{matrix}
-x-\sigma(x) & -\sigma(x)-\sigma^2(x) & \ldots &  -\sigma^{n-1}(x)-x\\
x-\sigma(x) & \sigma(x)-\sigma^2(x) &  \ldots & \sigma^{n-1}(x)-x\\
\vdots & \vdots & \ddots &  \vdots\\
\sigma^{n-2}(x)-\sigma^{n-1}(x) & \sigma^{n-1}(x)-x & \ldots & \sigma^{n-3}(x)-\sigma^{n-2}(x)
\end{matrix}\right)$$
that is a generator matrix of the lattice $\varphi_\beta(M)$, a rotated version of $D_n$.
\end{proof}

The lattice obtained above is a $D_n$-rotated, that is, a equivalent version of the lattice $D_n$. Since equivalent lattices have the same center density, the lattice $\varphi_\beta(M)$ has the better known center density in the dimensions $n=3$, $n=5$ and $n=37$, for example.

In the thesis \cite{grasi} and in the paper \cite{grasi2}, the authors produce $D_n$-rotated lattices for $n=\frac{p-1}{2}$, in which $p$ is a prime number. In this work we get $D_n$ for others values of $n$ not considered in the cited references above, like $n=7$, for example, since $2n+1=15$ is not a prime number.

\begin{ex}\label{ex1} Let's construct a rotated version of $\mathbb{Z}^7$ and of its sublattice $D_7$ having maximum diversity as proposed in this section. Consider $p=29$, that is congruent to $1$ module $7$. In this case, $r=2$ and $\lambda=1$. Thus,
$$\alpha = \zeta_{29}^{27} - \zeta_{29}^{26} - \zeta_{29}^{25} + \zeta_{29}^{24} + \zeta_{29}^{23} + \zeta_{29}^{22} + \zeta_{29}^{21} - \zeta_{29}^{20} + \zeta_{29}^{19} - \zeta_{29}^{18} - \zeta_{29}^{17} - \zeta_{29}^{16} +$$ $$+\zeta_{29}^{15} - \zeta_{29}^{14} - \zeta_{29}^{13} + \zeta_{29}^{12} - \zeta_{29}^{11} - \zeta_{29}^{10} - \zeta_{29}^{9} + \zeta_{29}^{8} - \zeta_{29}^{7} + \zeta_{29}^{6} + \zeta_{29}^{5} + \zeta_{29}^{4} + \zeta_{29}^{3} - \zeta_{29}^{2} - \zeta_{29} + 1$$
$$z=-2 \zeta_{29}^{27} - 4 \zeta_{29}^{26} - 2 \zeta_{29}^{25} - 2 \zeta_{29}^{24} - 2 \zeta_{29}^{23} - 4 \zeta_{29}^{21} - 2 \zeta_{29}^{19} - 2 \zeta_{29}^{18} - 4 \zeta_{29}^{17} - 2 \zeta_{29}^{15} - 4 \zeta_{29}^{14} - 2 \zeta_{29}^{12} -$$ $$-2\zeta_{29}^{11} - 4 \zeta_{29}^{10} - 4 \zeta_{29}^{8} - 2 \zeta_{29}^{7} - 2 \zeta_{29}^{6} - 2 \zeta_{29}^{5} - 2 \zeta_{29}^{3} - 4 \zeta_{29}^{2} - \zeta_{29} - 3$$
and
$$x=-3 \zeta_{29}^{27} - \zeta_{29}^{26} - \zeta_{29}^{25} - 3 \zeta_{29}^{24} - 3 \zeta_{29}^{23} - \zeta_{29}^{22} - \zeta_{29}^{21} - \zeta_{29}^{20} - \zeta_{29}^{19} + 3 \zeta_{29}^{18} + 3 \zeta_{29}^{16} - 3 \zeta_{29}^{15} - 3 \zeta_{29}^{14} +$$
$$+3 \zeta_{29}^{13} + 3 \zeta_{29}^{11} - \zeta_{29}^{10} - \zeta_{29}^{9} - \zeta_{29}^{8} - \zeta_{29}^{7} - 3 \zeta_{29}^{6} - 3 \zeta_{29}^{5} - \zeta_{29}^{4} - \zeta_{29}^{3} - 3 \zeta_{29}^{2} - 5.$$
So, the generator matrix of $\mathbb{Z}^7$ is given by
$$\frac{1}{29}\left(\begin{array}{ccccccc}
-19.747... & 4.729... & -13.016... & 2.244... & 2.387... & 7.991... & -13.588...\\
4.729... & -13.016... & 2.244... & 2.387... & 7.991... & -13.588... & -19.747...\\
-13.016... & 2.244... & 2.387... & 7.991... & -13.588... & -19.747... & 4.729...\\
2.244... & 2.387... & 7.991... & -13.588... & -19.747... & 4.729...& -13.016...\\
2.387... & 7.991... & -13.588... & -19.747... & 4.729...& -13.016... & 2.244...\\
7.991... & -13.588... & -19.747... & 4.729...& -13.016... & 2.244... & 2.387...\\
-13.588... & -19.747... & 4.729...& -13.016... & 2.244... & 2.387... & 7.991...
\end{array}\right)$$
while the generator matrix of $D_7$ is given by
$$\frac{1}{29}\left(\begin{array}{ccccccc}
15.017... & 8.286... & 10.772... & -4.631... & -10.378... & 5.597... & 33.335...\\
-24.477... & 17.746... & -15.260... & -0.143... & -5.603... & 21.579... & 6.158...\\
17.746... & -15.260... & -0.143... & -5.603... & 21.579... & 6.158... & -24.4777...\\
-15.260... & -0.143... & -5.603... & 21.579... & 6.158... & -24.4777... & 17.746...\\
-0.143... & -5.603... & 21.579... & 6.158... & -24.4777... & 17.746... & -15.260...\\
-5.603... & 21.579... & 6.158... & -24.4777... & 17.746... & -15.260... & -0.143...\\
21.579... & 6.158... & -24.4777... & 17.746... & -15.260... & -0.143... & -5.603...
\end{array}\right)$$
\end{ex}

Other advantage of the above construction (Theorem \ref{teorema1}) is the fact that the lattice $\varphi_\beta(M)\simeq D_n$ has maximum diversity, because it was obtained from a totally real number field $\mathbb{K}$. Since it has maximum diversity, we can calculate its minimum product distance in order to make this construction applicable to Rayleigh fading channels. The following section will be dedicated to the study of this value.

\section{Minimum product distance of $\mathbb{Z}^n$ and $D_n$, for any odd number $n>1$}\label{sec2}

When $N$ is a principal ideal, the minimum product distance of a lattice $\varphi_\beta(N)$ having maximum diversity depends only on the determinant of the lattice and on the discriminant of the field used to construct it. Therefore, next two results give conditions to calculate the minimum product distance of $\mathbb{Z}^n$ constructed in the previous section, with $n$ odd.

\begin{teorema}\label{teorema14} The $\mathbb{Z}$-module $I$ given in \ref{eq3} is an ideal of $\mathcal{O}_\mathbb{K}$.\end{teorema}
\begin{proof} Consider the $\mathbb{Z}$-module $J=\langle z,\sigma(z),\ldots,\sigma^{p-2}(z)\rangle_\mathbb{Z}$ in $\mathcal{O}_{\mathbb{Q}(\zeta_p)}=\mathbb{Z}[\zeta_p]$. We will see that $J$ is an ideal. For this, let $j=\sum_{i=0}^{p-2} a_i\sigma^i(z)$ be any element of $J$, where $a_i\in\mathbb{Z}$, $0\leq i\leq p-2$. We will show that $j\zeta_p \in J$. Of course, because of the item (b) of the Lemma \ref{lema1},
	$$j\zeta_p=\left(\sum_{i=0}^{p-2} a_i\sigma^i(z)\right)\zeta_p=\sum_{i=0}^{p-2} (-1)^ia_i\zeta_p\alpha(1-\zeta_p^{r^i})\zeta_p.$$
	Since $\{\zeta_p^{r^i}\}_{i=0}^{p-2}=\{\zeta_p^k\}_{k=1}^{p-1}$, we can reenumerate the above sum calling $b_k:=(-1)^ia_i$ such that $\zeta_p^{r^i}=\zeta_p^k$, for all $0\leq i\leq p-2$. Therefore,
	\begin{equation}\label{eq1}
	j\zeta_p=\zeta_p\alpha\sum_{k=1}^{p-1} b_k (1-\zeta_p^{k})\zeta_p
	\end{equation}
	Now, consider $c_1=-\sum_{k=1}^{p-1}b_k$ e $c_i=b_{i-1}$, $2\leq i\leq p-1$. Then:
	$$\sum_{i=1}^{p-1} c_i(1-\zeta_p^i)= -\sum_{k=1}^{p-1}b_k(1-\zeta_p) +\sum_{i=2}^{p-1} b_{i-1}(1-\zeta_p^i) =-b_{p-1}(1-\zeta_p)+ \sum_{i=1}^{p-2}b_i\left(-(1-\zeta_p)+(1-\zeta_p^{i+1}) \right)=$$
	$$=b_{p-1}(\zeta_p-1)+ \sum_{i=1}^{p-2}b_i\left(\zeta_p-\zeta_p^{i+1}\right)=b_{p-1}(1-\zeta_p^{p-1})\zeta_p+\sum_{i=1}^{p-2}b_i\left(1-\zeta_p^{i}\right)\zeta_p=\sum_{i=1}^{p-1}b_i\left(1-\zeta_p^{i}\right)\zeta_p.$$
	So, coming back on the equation \ref{eq1}, we have
	$$j\zeta_p=\zeta_p\alpha\sum_{i=1}^{p-1} c_i(1-\zeta_p^i)=\sum_{i=1}^{p-1} c_i\zeta_p\alpha(1-\zeta_p^i).$$
	Now, enumerate again the above sum putting the name $(-1)^kd_k$ for each term $c_i$ tal que $\zeta_p^i=\zeta_p^{r^k}$, $1\leq i \leq p-2$:
	$$j\zeta_p=\sum_{k=0}^{p-2} (-1)^kd_k\zeta_p\alpha(1-\zeta_p^{r^k})=\sum_{k=0}^{p-2}d_k\sigma^k\left(z\right).$$
	From this we conclude that $j\zeta_p\in J$ for all $j\in J$. By recurrence it follows that $j\zeta_p^k\in J$ for $0\leq k\leq p-2$. Since $\{\zeta_p^k\}_{k=0}^{p-2}$ is a $\mathbb{Z}$-basis for $\mathcal{O}_{\mathbb{Q}(\zeta_p)}$ then we can conclude that $j\mathcal{O}_{\mathbb{Q}(\zeta_p)}\subset J$, for all $j\in J$. Therefore, $J$ is an ideal in $\mathcal{O}_{\mathbb{Q}(\zeta_p)}$. In turn, observe that $I$ coincides with the ideal $Tr_{\mathbb{Q}(\zeta_p)/\mathbb{K}}(J)$, because $x=Tr_{\mathbb{Q}(\zeta_p)/\mathbb{K}}(z)$ and because for each $\sigma^i(z)$ of the $\mathbb{Z}$-basis of $J$ (since there exists $q$ and $r$ such that $i=qn+r$, $0\leq r<n$, and since $\sigma^n(x)=x$),
	$$Tr_{\mathbb{Q}(\zeta_p)/\mathbb{K}}(\sigma^i(z))=\sigma^i(x)=\sigma^{qn+r}(x)=\sigma^r(x)\in I.$$
	Therefore, $I$ is an ideal in $\mathcal{O}_\mathbb{K}$.
\end{proof}

Of the fact that $I$ is a principal ideal it follows that
\begin{equation}\label{eq2}
x\mathcal{O}_\mathbb{K}\subset I.
\end{equation}

\begin{teorema}\label{teorema2}
If $\sigma(x)/x\in\mathbb{Z}[\zeta_p]$ then $I$ is the principal ideal of $\mathcal{O}_\mathbb{K}$ generated by $x$.
\end{teorema}
\begin{proof}
Since $\sigma(x)/x\in \mathbb{Z}[\zeta_p]\cap \mathbb{K}$ then $u:=\sigma(x)/x\in\mathcal{O}_\mathbb{K}$. So $\sigma^i(u)\in\mathcal{O}_\mathbb{K}$ and, from this,
$$\{x,\sigma(x),\sigma^2(x),\ldots,\sigma^{n-1}(x)\}=\{x,ux,\sigma(u)x,\ldots,\sigma^{n-2}(u)x\}\subset x\mathcal{O}_\mathbb{K}$$
that is, $I\subset x\mathcal{O}_\mathbb{K}$. From this and from Equation \ref{eq2} follows that $x\mathcal{O}_\mathbb{K}=I$.
\end{proof}

\begin{corol} \label{corol1}
If $\sigma(x)/x\in\mathbb{Z}[\zeta_p]$ then the minimum product distance of the lattice $\varphi_\beta(I)$ (equivalent to $\mathbb{Z}^n$) is equal to $p^{\frac{1-n}{2}}.$ Besides that,
$|N_\mathbb{K}(x)|=p^{\frac{n+1}{2}}.$
\end{corol}
\begin{proof}
Due to the Theorem \ref{teorema2}, $I$ is the principal ideal in $\mathcal{O}_\mathbb{K}$ generated by $x$. Because it is a principal ideal, the Theorem 1 of \cite{bayer-oggier} implies that $d_{p,min}(\varphi_\beta(I))=\sqrt{D/d_\mathbb{K}}$
where $D$ is the determinant of the lattice and $d_\mathbb{K}$ is the discriminant of $\mathbb{K}$. Since the lattice is $\mathbb{Z}^n$-rotated then $D=1$. Now, note that the smallest cyclotomic field containing $\mathbb{K}$ is $\mathbb{Q}(\zeta_p)$. In fact, if there was a integer number $l$ such that $\mathbb{K}\subset\mathbb{Q}(\zeta_l)$ where $l<p$ then
$$\mathbb{K}\subset \mathbb{Q}(\zeta_p)\cap \mathbb{Q}(\zeta_l) = \mathbb{Q}(\zeta_{mdc(p,l)})=\mathbb{Q}(\zeta_1)=\mathbb{Q}\Longrightarrow \mathbb{K}=\mathbb{Q}$$
what is a contradiction because the degree of $\mathbb{K}$ is $n>1$. Because of this, we say that $p$ is the \textit{conductor} of $\mathbb{K}$. From this, due to \cite{trajano} (Corollary 4.2), we conclude that $d_\mathbb{K}=p^{n-1}$. Therefore,
$d_{p,min}(\varphi_\beta(I))=\sqrt{1/p^{n-1}}=p^{\frac{1-n}{2}}$. On the other hand, we know that the minimum product distance is equal to $\sqrt{N_\mathbb{K}(\beta)}\min_{0\neq y\in I} |N_\mathbb{K}(y)|$. If $\beta=1/p^2$ then $\sqrt{N_\mathbb{K}(\beta)}=(1/p)^n$. Besides that, since $I$ is principal generated by $x$, the value $\min_{0\neq y\in I} |N_\mathbb{K}(y)|$ is reached by $x$, that is,
$$|N_\mathbb{K}(x)|=\min_{0\neq y\in I}|N_\mathbb{K}(y)|=\frac{d_{p,min}(\varphi_\beta(I))}{\sqrt{N_\mathbb{K}(\beta)}}=p^{\frac{1-n}{2}}p^n=p^{\frac{n+1}{2}}.$$
\end{proof}

Observe that the Theorem \ref{teorema2} and the Corollary \ref{corol1} need the hypothesis that $\sigma(x)/x$ is an algebraic integer (belongs to $\mathbb{Z}[\zeta_p]$). The following Theorem guarantees that this is true when $\mathbb{K}$ is the maximal real cyclotomic subfield $\mathbb{Q}(\zeta_p+\zeta_p^{-1})$ in $\mathbb{Q}(\zeta_p)$. Note that this is a case treated in \cite{grasi2}.

\begin{teorema} If $(p-1)/n=2$ then $\sigma(x)/x \in \mathbb{Z}[\zeta_p]$. \label{teorema4}
\end{teorema}
\begin{proof}
Due to the Lemma \ref{lema1} and to the definition of $x$ we have:
$$\frac{\sigma(x)}{x}=\frac{-\zeta_p^\lambda\alpha\left(-(1-\zeta_p^{r^{n+1}})+(1-\zeta_p^{r^{2n+1}})\right)}{\zeta_p^\lambda\alpha\left(-(1-\zeta_p^{r^{n}})+(1-\zeta_p^{r^{2n}})\right)}=\frac{\zeta_p^{r}-\zeta_p^{r^{n+1}}}{\zeta_p^{r^n}-\zeta_p}=-\frac{\zeta_p^{r}}{\zeta_p}\frac{\left(1-\zeta_p^{r^{n+1}-r}\right)}{\left(1-\zeta_p^{r^n-1}\right)}=-\zeta_p^{r-1}\frac{\left(1-\zeta_p^{r(r^{n}-1)}\right)}{\left(1-\zeta_p^{r^n-1}\right)}.$$
Follows from the Lemma 1.3 of \cite{wash} that the last term is an unit in $\mathcal{O}_\mathbb{K}$. In particular, $\sigma(x)/x$ is an element of $\mathcal{O}_\mathbb{K}\subset\mathbb{Z}[\zeta_p]$, as we wanted to show.
\end{proof}

Therefore, when $p=2n+1$ is a prime number, we can guarantee that there exists $\mathbb{Z}^n$ having maximum diversity and having minimum product distance equal to $p^{\frac{1-n}{2}}$. This occurs, for example, for $n=3,5,9,11,15,19,...$. However, we also can calculate the minimum product distance in other cases only verifying if $\sigma(x)/x$ is an algebraic integer, as in the following example:

\begin{ex}\label{ex2} Consider the rotated $\mathbb{Z}^7$ ($n=7$) developed in the Example \ref{ex1}. In this case, the smallest prime number able to be used to make the construction is $p=29$, which does not satisfy the equality $p=2n+1$. Therefore, to apply the Corollary \ref{corol1} we need to calculate the quocient $\sigma(x)/x$ and verify if this number is an algebraic integer. In fact, $\sigma(x)/x = -\zeta_{29}-\zeta_{29}^{12}-\zeta_{29}^{17}-\zeta_{29}^{28}$
belongs to $\mathbb{Z}[\zeta_{29}]$, because it is a integer combination of powers of $\zeta_{29}$. Therefore, the Corollary \ref{corol1} guarantees that the minimum product distance of this lattice is
$p^{\frac{1-n}{2}}=1/29^{3}$ and that $|N_\mathbb{K}(x)|=p^{\frac{n+1}{2}}=29^4$.
\end{ex}

\begin{obs}\label{obs3} The hypothesis $\sigma(x)/x\in\mathbb{Z}[\zeta_{p}]$ can not be discarted in the Theorem \ref{teorema2}. For example, when $n=13$ and $p=131$ (or $p=157$, or $p=313$), the quotient $\sigma(x)/x$ is not an algebraic integer. However, if $p=53$ or $p=79$, this quotient belongs to $\mathbb{Z}[\zeta_{p}]$.
\end{obs}

Now we will study the minimum product distance of lattices $D_n$ constructed on the Theorem \ref{teorema1} using the $\mathbb{Z}$-module $\{x+\sigma(x),x-\sigma(x),\sigma(x)-\sigma^2(x),\ldots,\sigma^{n-2}(x)-\sigma^{n-1}(x)\}$.

If $M$ was a principal ideal, we could conclude that the minimum product distance of $\varphi_{1/p²}(M)$ was $2p^{\frac{1-n}{2}}$ (see \cite{grasi}, Section 4.4). However, we will see several situations in which this value is equal to $p^{\frac{1-n}{2}}$. Therefore, $M$ can not be a principal ideal in these cases.

\begin{corol}\label{corol3}
	If $u=\sigma(x)/x\in\mathbb{Z}[\zeta_p]$ then the minimum product distance of $\varphi_{1/p²}(M)\simeq D_n$ satisfies $d_{p,min}(\varphi_{1/p²}(M))\geq p^{\frac{1-n}{2}}.$
\end{corol}
\begin{proof} This follows straight from the Corollary \ref{corol1} because $\varphi_{1/p²}(M)$ is a sublattice of $\varphi_{1/p^2}(I)$.
\end{proof}

\begin{corol} \label{teorema3}
If $u=\sigma(x)/x\in\mathbb{Z}[\zeta_p]$ and if $1+u$ or $1-u$ is an unit in $\mathcal{O}_\mathbb{K}$ then the minimum product distance of $\varphi_{1/p²}(M)\simeq D_n$ is equal to $p^{\frac{1-n}{2}}$.
\end{corol}
\begin{proof} If $1+u$ is an unit in $\mathcal{O}_\mathbb{K}$ then from the Corollary \ref{corol1} follows that $x+\sigma(x)=x(1+u)$ has absolute value of its norm given by $|N_\mathbb{K}(x+\sigma(x))|=|N_\mathbb{K}(x)|=p^{\frac{n+1}{2}}$. The same argument can be used for the case in which $1-u$ is an unit.
On the one hand, since $d_{p,min}(\varphi_{1/p²}(M))\geq p^{\frac{1-n}{2}}$ (Corollary \ref{corol3}), we have $\min_{0\neq y \in M} |N_\mathbb{K}(y)| \geq N_\mathbb{K}(1/p^2)^{-2}p^{\frac{1-n}{2}} = p^{\frac{n+1}{2}}.$ Since one of the values $y=x\pm\sigma(x)$ achieve this minimum, then
$d_{p,min}(\varphi_{1/p^2}(M))=p^{-n}p^{\frac{n+1}{2}}=p^{\frac{1-n}{2}}$.
\end{proof}

\begin{obs} Observe that the fact of $1\pm u$ be an unit in $\mathcal{O}_\mathbb{K}$ on the hypothesis of the Corollary \ref{teorema3} is equivalent to the equality $|N_\mathbb{K}(x\pm\sigma(x))|=|N_\mathbb{K}(x)|$. \label{obs1}
\end{obs}

Remember of the Theorem \ref{teorema4} that $u\in\mathcal{O}_\mathbb{K}$ when $(p-1)/n=2$. Therefore, for that the minimum product distance in these cases be equal to $p^{\frac{1-n}{2}}$ we only need verify if $1+u$ or if $1-u$ is an unit. The following result presents a particular case where the hypothesis of the above theorem always occurs:

\begin{corol} If $n=3$ and $u\in\mathbb{Z}[\zeta_p]$ then the minimum product distance of $D_3$ constructed via any $p\equiv 1~(mod~3)$ is $1/p$. \label{corol11}
\end{corol}
\begin{proof} Since $n=3$ then $M=\langle x,\sigma(x),\sigma^2(x)\rangle_\mathbb{Z}$. The Lemma \ref{lema2} implies that $x\sigma(x)+\sigma(x)\sigma^2(x)+\sigma^2(x)x=0$. Using this identity we have
$$N_\mathbb{K}(x+\sigma(x))=(x+\sigma(x))(\sigma(x)+\sigma^2(x))(\sigma^2(x)+x)=$$
$$(x^2+x\sigma(x)+\sigma(x)\sigma^2(x)+\sigma^2(x)x)(\sigma(x)+\sigma^2(x))=$$
$$x^2(\sigma(x)+\sigma^2(x))=x(x\sigma(x)+x\sigma^2(x))=-x(\sigma(x)\sigma^2(x))=-N_\mathbb{K}(x).$$
The conclusion follows from the Remark \ref{obs1} and from the Corollary \ref{teorema3}.
\end{proof}

\begin{ex} Consider the rotated $D_7$ ($n=7$) developed in the Example \ref{ex1}. The prime number used was $p=29$. In the Example \ref{ex2} was shown that $\sigma(x)/x$ is an algebraic integer. Besides that, $1-u = 1+\zeta_{29}+\zeta_{29}^{12}+\zeta_{29}^{17}+\zeta_{29}^{28}$
and
$$\frac{1}{1-u}=\zeta_{29}^{27} - 2 \zeta_{29}^{26} - 2 \zeta_{29}^{25} + \zeta_{29}^{24} - \zeta_{29}^{23} - 2 \zeta_{29}^{22} - 2 \zeta_{29}^{19} - \zeta_{29}^{18} - \zeta_{29}^{16} - \zeta_{29}^{15} - \zeta_{29}^{14} -$$
$$- \zeta_{29}^{13} - \zeta_{29}^{11} - 2 \zeta_{29}^{10} - 2 \zeta_{29}^{7} - \zeta_{29}^{6} + \zeta_{29}^{5} - 2 \zeta_{29}^{4} - 2 \zeta_{29}^{3} + \zeta_{29}^{2} - 3.$$
So $1-u$ is an unit in $\mathbb{Z}[\zeta_p]$ and, so, it is an unit in $\mathcal{O}_\mathbb{K}$. Follows from the Corollary \ref{teorema3} that the minimum product distance of $\varphi_{1/29^2}(M)\simeq D_7$ is equal to $1/29^3$.\\
However, note that $1+u$ is not an unit in $\mathcal{O}_\mathbb{K}$ because
$$\frac{1}{1+u}=\frac{5}{17} \zeta_{29}^{27} + \frac{6}{17} \zeta_{29}^{26} + \frac{2}{17} \zeta_{29}^{25} + \frac{5}{17} \zeta_{29}^{24} + \frac{7}{17} \zeta_{29}^{23} + \frac{6}{17} \zeta_{29}^{22} + \frac{2}{17} \zeta_{29}^{21} + \frac{2}{17} \zeta_{29}^{20} + \frac{2}{17} \zeta_{29}^{19} - \frac{1}{17} \zeta_{29}^{18} - \frac{1}{17} \zeta_{29}^{16} + \frac{7}{17} \zeta_{29}^{15} +$$
$$+ \frac{7}{17} \zeta_{29}^{14} - \frac{1}{17} \zeta_{29}^{13} - \frac{1}{17} \zeta_{29}^{11} + \frac{2}{17} \zeta_{29}^{10} + \frac{2}{17} \zeta_{29}^{9} + \frac{2}{17} \zeta_{29}^{8} + \frac{6}{17} \zeta_{29}^{7} + \frac{7}{17} \zeta_{29}^{6} + \frac{5}{17} \zeta_{29}^{5} + \frac{2}{17} \zeta_{29}^{4} + \frac{6}{17} \zeta_{29}^{3} + \frac{5}{17} \zeta_{29}^{2} + \frac{7}{17}$$
is not an algebraic integer.
\end{ex}

\begin{obs}\label{obs2} The hypothesis that $1+u$ or $1-u$ is an unit can not be discarted in the Corollary \ref{teorema3}. As in the Remark \ref{obs3}, if $n=13$ and $p=53$ or $p=79$ then $\sigma(x)/x\in\mathbb{Z}[\zeta_p]$. However, in these two cases $1+u$ and $1-u$ are not invertible in $\mathcal{O}_\mathbb{K}$. However, note that this is not a problem because the Corollary \ref{corol3} can still be applied.
\end{obs}

\section{Algebraic lattices $\mathbb{Z}^n$ and $D_n$, for any even number $n>1$}\label{sec4}

In this section we construct algebraic lattices $\mathbb{Z}^n$ and $D_n$ for any even integer number $n>1$ and study their minimum product distance. Firstly, we redeem the construction of $\mathbb{Z}^k$ for $k=2^m>1$ of \cite{sethoggier} and obtain $D_k$. After, combining these results with the construction and results made in Sections \ref{sec1} and \ref{sec2} for $\mathbb{Z}^l$, $l$ odd, we obtain $\mathbb{Z}^n$ and $D_n$, for any $n>1$.

Let $m\geq 3$ be an integer number. Consider $k=2^{m-2}$ and $\omega=e^{\frac{2\pi i}{2^m}}$ a $2^m$-th primitive root of unity. Denoting by $\theta$ the number $\omega+\omega^{-1}$ we see that $\mathbb{L}=\mathbb{Q}(\theta)$ is the maximal real subfield of the cyclotomic field $\mathbb{Q}(\omega)$. So, since $[\mathbb{Q}(\omega):\mathbb{L}]=2$, the degree of the field $\mathbb{L}$ is $k$. Besides that, $\mathcal{O}_\mathbb{L}=\mathbb{Z}[\theta]$ (see \cite{wash}, Proposition 2.16). Denote $\theta_j=\omega^j+\omega^{-j}$, for $j\geq 0$. The next theorem presents an algebraic $\mathbb{Z}^k$-rotated lattice:

\begin{teorema}[\cite{sethoggier}, Theorem 1] Consider $w_0=1$, $w_1=1+\theta_1$, $\ldots$, $w_{k-1}=1+\theta_1+\ldots+\theta_{k-1}$. The set $H=\{w_0,w_1,\ldots,w_{k-1}\}$ is a basis of $\mathcal{O}_\mathbb{L}$. Besides that, if $\beta=1/k-\theta/(2k)$ then the lattice $\varphi_\beta(\mathcal{O}_\mathbb{L})$ is a rotated version of $\mathbb{Z}^k$ with this basis. \label{teorema7}
\end{teorema}

Denote by $\tau$ the generator of the cyclic Galois group of $\mathbb{L}$ over $\mathbb{Q}$. If $r$ is the primitive element module $2^{m-1}$, the monomorphism $\tau$ can be defined by $\tau(\omega)=\omega^r$. The generator matrix of $\varphi_\beta(\mathcal{O}_\mathbb{L})$ is, by definition,
\begin{equation}\label{eq5}
G=\left(\begin{matrix}
w_0 & \tau(w_0) & \ldots & \tau^{k-1}(w_0)\\
w_1 & \tau(w_1) & \ldots & \tau^{k-1}(w_1)\\
\vdots & \vdots & \ddots & \vdots\\
w_{k-1} & \tau(w_{k-1}) & \ldots & \tau^{k-1}(w_{k-1})
\end{matrix}\right)\left(\begin{matrix}
\sqrt{\beta} & 0 & \ldots & 0\\
0 & \sqrt{\tau(\beta)} & \ldots & 0\\
\vdots & \vdots & \ddots & \vdots\\
0 & 0 & \ldots & \sqrt{\tau^{k-1}(\beta)}
\end{matrix}\right).
\end{equation}
To calculate explicitly this matrix, note that $\tau^j(w_0)=1$ and that $\tau^j(w_i) = 1+\theta_{r^j}+\theta_{2r^j}+\ldots+\theta_{ir^j}$, for $0\leq i,j\leq k-1$. Consequently it is possible construct the lattice $D_k$ analogously to what was done in the Section 4.2 of \cite{grasi2}:

\begin{teorema} Consider $L$ the $\mathbb{Z}$-module generated by
	$$\{w_0+w_1,w_0-w_1,\ldots,w_{k-2}-w_{k-1}\}=\{2+\theta_1,-\theta_1,-\theta_2,\ldots,-\theta_{k-1}\}.$$
	So, the algebraic lattice $\sigma_\beta(L)$ is a rotated version of $D_k$ (with $k=2^{m-2}$). \label{teorema5}
\end{teorema}
\begin{proof} To prove this theorem it is enough multiply the matrix in \ref{eq6} by the above matrix $G$, as in the proof of Theorem \ref{teorema1}. Since $G$ is a rotation matrix, this lattice is a rotated version of $D_k$.
\end{proof}

Note that other set of generators for the $\mathbb{Z}$-module $L$ in the Theorem \ref{teorema5} is given by
\begin{equation}\label{eq7}
\{2\theta_0,\theta_1,\theta_2,\ldots,\theta_{k-1}\}.
\end{equation}
So, it is not difficult prove that the $\mathbb{Z}$-module $L$ is the principal ideal $\theta_1\mathcal{O}_\mathbb{L}$ (see \cite{grasi2}, Proposition 4.7). To the following we need the next lemma:

\begin{lema}[\cite{lopes}, Theorem 3.2]\label{lema3} The discriminant of $\mathbb{L}$ is given by $2^{(m-1)k-1}$.
\end{lema}

Since the lattices $\mathbb{Z}^k$ and $D_k$ ($k=2^{m-2}$, $m\geq 3$) are constructed via the twisted homomorphism using principal ideals in ring of integers of totally real number fields then both have maximum diversity. Besides that, due to the fact that ideals used in the construction are principal and due to the Lemma \ref{lema3}, the minimum product distance of $\mathbb{Z}^k$ is
$$d_{p,min}(\varphi_\beta(\mathcal{O}_\mathbb{L}))=\sqrt{\frac{det(\varphi_\beta(\mathcal{O}_\mathbb{L}))}{d_\mathbb{L}}}=\sqrt{\frac{1}{2^{(m-1)k-1}}}=2^{\frac{1-(m-1)k}{2}}$$
and the minimum product distance of $D_k$ is
$$d_{p,min}(\varphi_\beta(L))=\sqrt{\frac{det(\varphi_\beta(L))}{d_\mathbb{L}}}=\sqrt{\frac{4}{2^{(m-1)k-1}}}=2^{\frac{3-(m-1)k}{2}}.$$

\begin{ex} Consider $m=3$ and $k=2$. So $\theta=\zeta_8+\zeta_8^{-1}=\sqrt{2}$, $\mathbb{L}=\mathbb{Q}(\sqrt{2})$ and $\beta=(2-\sqrt{2})/4$. Now, $H=\{w_0=1,w_1=1+\sqrt{2}\}$ is a basis of the lattice $\varphi_\beta(\mathbb{Z}[\sqrt{2}])\simeq \mathbb{Z}^2$. In turn, $L=\langle 2,\sqrt{2} \rangle$ is the $\mathbb{Z}$-module such that $\varphi_\beta(L)\simeq D_2$. Both lattices have maximum diversity. In this construction, the minimum product distance of $\mathbb{Z}^2$ is $1/2\sqrt{2}$ and of $D_2$ is $1/\sqrt{2}$.
\end{ex}

\begin{ex} Consider $m=4$, $k=4$,  $\theta=\zeta_{16}+\zeta_{16}^{-1}=2cos(\pi/8)$, $\mathbb{L}=\mathbb{Q}(2cos(\pi/8))$ and $\beta=(1-cos(\pi/8))/4$. Now, $H=\{w_0=1,w_1=1+2cos(\pi/8),w_2=1+2cos(\pi/8)+\sqrt{2},1+2cos(\pi/8)+\sqrt{2}+2cos(3\pi/8)\}$ is a basis of the lattice $\varphi_\beta(\mathbb{Z}[2cos(\pi/8)])\simeq \mathbb{Z}^4$. In turn, $L=\langle 2, 2cos(\pi/8),\sqrt{2},2cos(3\pi/8) \rangle$ is the $\mathbb{Z}$-module such that $\varphi_\beta(L)\simeq D_4$. Both lattices have maximum diversity. In this construction, the minimum product distance of $\mathbb{Z}^4$ is $2^{-11/2}$ and of $D_4$ is $2^{-9/2}$.
\end{ex}

Now we will construct the lattices $\mathbb{Z}^n$ and $D_n$, for any even $n>1$, using the compositum of the field considered above to construct lattices $\mathbb{Z}^n$ and $D_n$ for $n$ a power of two with the field used in Section \ref{sec1} to construct $\mathbb{Z}^n$ and $D_n$ for $n$ an odd number, following the ideas and constructions made in \cite{bayer-oggier} and \cite{sethoggier}. For this, we use notations made above and in the Section \ref{sec1}. Consider $l>1$ an odd number and $k=2^{m-2}$, $m\geq 3$. Remember that $\mathbb{K}$ was defined in \ref{eq4} and has odd degree $l$, while $\mathbb{L}=\mathbb{Q}(\omega)=\mathbb{Q}(\zeta_{2^m}+\zeta_{2^m}^{-1})$ has degree $k=2^{m-2}$. We can see that $\mathbb{K}\cap\mathbb{L}=\mathbb{Q}$.

Lattices $\mathbb{Z}^n$ and $D_n$ for $n$ an odd number and for $n$ equal a power of $2$ were treated previously. Only remains the case $n=kl$, where $k$ and $l$ are as above. For this, let $\mathbb{KL}$ be the compositum of the fields $\mathbb{K}$ and $\mathbb{L}$ (this is, the smallest field containing $\mathbb{K}$ and $\mathbb{L}$ simultaneously).  In this way, the compositum $\mathbb{K}\mathbb{L}$ has degree $n=kl$ and can be used to construct $\mathbb{Z}^n$ and $D_n$ (see Chapter 13, Item W, of \cite{rib}).

Next theorem is similar to the Proposition 6 of \cite{bayer-oggier}. Its proof can be done following the steps of the Section 4 of \cite{sethoggier}. Consider the ideal $I$ of \ref{eq3}, $J=\mathcal{O}_\mathbb{L}$, $p$ the prime number satisfying $p\equiv 1~(mod~l)$ used previously and $w_i$ defined above.

\begin{teorema}[\cite{sethoggier}, Section 4] \label{teorema8} Let $\mathcal{I}$ be the product ideal $IJ\subset\mathcal{O}_\mathbb{KL}$ and $\beta=p^{-2}(1/k-\theta/2k)$. So
\begin{enumerate}
	\item[a)] $\mathcal{I}=\langle w_0x,w_0\sigma(x)\ldots,w_0\sigma^{l-1}(x),w_1x,\ldots,w_{1}\sigma^{l-1}(x),\ldots,w_{k-1}x,\ldots,w_{k-1}\sigma^{l-1}(x)\rangle_\mathbb{Z}.$
	\item[b)] $\varphi_\beta(\mathcal{I})$ is a rotated version of $\mathbb{Z}^n$.
	\item[c)] The generator matrix of this lattice is the tensor product of the matrix $G$ explicited in the Theorem \ref{teorema6} by the matrix $G$ explicited in \ref{eq5}.
\end{enumerate}
\end{teorema}

Analogously to what was done above, it is possible take out a rotated version of $D_n$ from $\mathbb{Z}^n$:

\begin{teorema}\label{teorema9} Consider the $\mathbb{Z}$-module	$\mathcal{M}$ generated by
	$$\{w_0(x+\sigma(x))\}\cup \{w_{i}\sigma^{n-1}(x)-w_{i+1}x, 0\leq i\leq k-2\}\cup\left(\bigcup_{0\leq j \leq k-1}\{w_j(\sigma^i(x)-\sigma^{i+1}(x)),0\leq i\leq n-2\} \right)$$
and $\beta = p^{-2}(1/k-\theta/(2k))$. So, the algebraic lattice $\varphi_\beta(\mathcal{M})$ is a rotated version of $D_n$.
\end{teorema}
\begin{proof} It is enough to multiply the matrix in \ref{eq6} (of dimension $n\times n$) by the matrix of the item (c) of the Theorem \ref{teorema8}, as in the proof of the Theorem \ref{teorema1}. Since the last matrix is a rotation matrix, then the lattice is a rotated version of $D_n$.
\end{proof}

In the Theorems \ref{teorema8} and \ref{teorema9} were constructed rotated versions of $\mathbb{Z}^n$ and $D_n$. They were constructed via $\mathbb{Z}$-modules of the compositum $\mathbb{KL}$. Since $\mathbb{K}$ and $\mathbb{L}$ are totally real number fields then $\mathbb{KL}$ is too. Therefore, the lattices $\mathbb{Z}^n$ and $D_n$ have maximum diversity. So, we can calculate the minimum product distance of each one of these lattices restricted to some conditions using known results for the cases $k=2^{m-2}$ ($m\geq 3$) and $l$ odd. In the following, consider $\beta=p^{-2}\left(1/k-\theta/(2k)\right)$.

\begin{lema}[\cite{rib}, Chapter 14, Item W] \label{lema10} The discriminant of the compositum $\mathbb{K}\mathbb{L}$ is $d_\mathbb{K}^kd_\mathbb{L}^{l}$.
\end{lema}

\begin{teorema}\label{teorema15} If $\sigma(x)/x\in\mathbb{Z}[\zeta_p]$ then the ideal $\mathcal{I}=I\mathcal{O}_\mathbb{L}$ is principal generated by $x$ and the minimum product distance of the lattice $\varphi_\beta(\mathcal{I})\simeq\mathbb{Z}^n$ is given by $p^{\frac{k-n}{2}}2^{\frac{l-n(m-1)}{2}}$. Besides that,
\begin{equation}\label{eq8}
|N_\mathbb{KL}(x)|=p^{\frac{n+k}{2}}.
\end{equation}
\end{teorema}
\begin{proof} Due to the Theorem \ref{teorema2}, $I=x\mathcal{O}_\mathbb{K}$. This implies that $x\mathcal{O}_{\mathbb{K}\mathbb{L}}=x\mathcal{O}_\mathbb{K}\mathcal{O}_\mathbb{L}=I\mathcal{O}_\mathbb{L}$, this is, $\mathcal{I}=I\mathcal{O}_\mathbb{L}$ is a principal ideal generated by $x$. Now, of the Lemma \ref{lema10} follows that
$$d_{p,min}(\varphi_\beta(\mathcal{I}))=\sqrt{\frac{D}{d_{\mathbb{KL}}}}=\frac{1}{\sqrt{d_\mathbb{K}^kd_\mathbb{L}^l}}=p^{\frac{(1-l)k}{2}}2^{\frac{l-l(m-1)k}{2}}.$$
Finally, since $x\in\mathbb{K}$ and  $|N_\mathbb{K}(x)|=p^{\frac{l+1}{2}}$ (Corollary \ref{corol1}), follows from the transitivity propertie of the norm that $|N_\mathbb{KL}(x)|=|N_\mathbb{K}\left(N_{\mathbb{KL}/\mathbb{K}}(x)\right)|=|N_\mathbb{K}(x)|^k=p^{\frac{(l+1)k}{2}}=p^{\frac{n+k}{2}}.$
\end{proof}

\begin{teorema}\label{teorema11} If $u=\sigma(x)/x\in\mathbb{Z}[\zeta_p]$ then the minimum product distance of $\varphi_{\beta}(\mathcal{M})\simeq D_n$ satisfies
	$$d_{p,min}(\varphi_\beta(\mathcal{M}))\geq 2^{\frac{l-n(m-3)}{2}}p^{\frac{k-n}{2}}$$
\end{teorema}
\begin{proof} Since $\mathcal{I}$ is a principal ideal generated by $x$ and using \ref{eq8},
$$\min_{0\neq y\in \mathcal{M}} |N_\mathbb{KL}(y)| \geq \min_{0\neq y\in \mathcal{I}} |N_\mathbb{KL}(y)|= |N_\mathbb{KL}(x)|=p^{\frac{n+k}{2}}.$$
Now, $N_\mathbb{KL}(\beta)=p^{-2n}(2k)^{-n}N_\mathbb{KL}(2-\theta)$. Since $2-\theta = 2-\omega-\omega^{-1}=(1-\omega)(1-\omega^{-1})$ then
$$N_{\mathbb{Q}(\omega)}(2-\theta)=N_{\mathbb{Q}(\omega)}\left((1-\omega)(1-\omega^{-1})\right)=\left(N_{\mathbb{Q}(\omega)}(1-\omega)\right)^2=2^2$$
and, using the transitivity propertie of the norm, $2^2=N_{\mathbb{Q}(\omega)}(2-\theta)=N_\mathbb{L}\left(N_{\mathbb{Q}(\omega)/\mathbb{L}}(2-\theta)\right)=N_\mathbb{L}(2-\theta)^2$, which implies that $N_\mathbb{L}(2-\theta)=2$. So, again using the transitivity propertie of the norm,
$$N_\mathbb{KL}(2-\theta)=N_\mathbb{L}\left(N_{\mathbb{KL}/\mathbb{L}}(2-\theta)\right)=N_\mathbb{L}(2-\theta)^l=2^l.$$
Therefore,
$$d_{p,min}(\varphi_\beta(\mathcal{M}))=\sqrt{N_\mathbb{KL}(\beta)}\min_{0\neq y\in\mathcal{M}} |N_\mathbb{KL}(y)|\geq p^{-n}(2k)^{-n/2}2^{l/2}p^{\frac{n+k}{2}}=2^{\frac{l-n(m-3)}{2}}p^{\frac{k-n}{2}}.$$
\end{proof}

\begin{corol}\label{teorema10} If $u=\sigma(x)/x\in\mathbb{Z}[\zeta_p]$ and if $1+u$ or $1-u$ is an unit in $\mathcal{O}_\mathbb{K}$ then the minimum product distance of $\varphi_{\beta}(\mathcal{M})\simeq D_n$ is equal to $2^{\frac{l-n(m-3)}{2}}p^{\frac{k-n}{2}}$.
\end{corol}
\begin{proof} This proof will be made considering that $1+u$ is an unit. The case in which $1-u$ is an unit can be done analogously. If $1+u$ is an unit in $\mathcal{O}_\mathbb{LK}$ (because it is an unit in $\mathcal{O}_\mathbb{K}$) then $|N_\mathbb{KL}(x+\sigma(x))|=|N_\mathbb{KL}(x(1+u))|=|N_\mathbb{KL}(x)|.$ Note that the first element of the set of generators of $\mathcal{M}$ enunciated in the Theorem \ref{teorema9} is $w_0(x+\sigma(x))=x+\sigma(x)$ (because $w_0=1$). So, by \ref{eq8}, we have
$$\min_{0\neq y\in \mathcal{M}} |N_\mathbb{KL}(y)|\leq |N_\mathbb{KL}(w_0(x+\sigma(x)))|=|N_\mathbb{KL}(x)|.$$
This proves that $d_{p,min}(\varphi_\beta(\mathcal{M}))\leq 2^{\frac{l-n(m-3)}{2}}p^{\frac{k-n}{2}}$ as in the proof of the Theorem \ref{teorema11}. Finally, the Theorem \ref{teorema11} also concludes this proof.
\end{proof}

\begin{ex} To construct $\mathbb{Z}^{14}$ ($n=2.7$, $k=2$, $l=7$) we will use the field $\mathbb{K}$ constructed implicitily in the Example \ref{ex1} and the field $\mathbb{L}=\mathbb{Q}(\sqrt{2})$. The ideal that produces a rotated version of $\mathbb{Z}^{14}$ is $\mathcal{I}= \langle x,\sigma(x),\ldots,\sigma^6(x),(1+\sqrt{2})x,(1+\sqrt{2})\sigma(x),\ldots,(1+\sqrt{2})\sigma^6(x)\rangle$, with $\beta=29^{-2}(2-\sqrt{2})/4$. It has maximum diversity and its minimum product distance is $2^{-21/2}29^{-6}$. In turn, $\mathcal{M}=\langle x+\sigma(x),x-\sigma(x),\ldots,\sigma^{5}(x)-\sigma^{6}(x),\sigma^6(x)-(1+\sqrt{2})x,(1+\sqrt{2})(x-\sigma(x)),\ldots,(1+\sqrt{2})(\sigma^{5}(x)-\sigma^{6}(x)) \rangle$ generates $D_{14}$ having maximum diversity and minimum product distance given by $2^{7/2}29^{-6}$, because the hypothesis of the Corollary \ref{teorema10} is true in this example.
\end{ex}

\section{Comparisons and conclusions}\label{sec5}

In this paper we worked in parallel with rotated versions of $\mathbb{Z}^n$ and $D_n$. Comparing these lattices with respect to packing density, it is a known fact that $D_n$ has better perfomance than $\mathbb{Z}^n$ when $n>2$, because the center density of $D_n$ is $2^{-(n+2)/2}$, while the center density of $\mathbb{Z}^n$ is equal to $2^{-n}$. With respect to minimum product distance of the lattices constructed here, to make a fair comparison, it is necessary that $\mathbb{Z}^n$ and its sublattice $D_n$ have same minimum norm or same volume. Since this not occurs, we need define comparative forms of the minimum product distance:

\begin{defin} (1) The relative minimum product distance $d_{p,rel}$ of a complete lattice $\Lambda\subset\mathbb{R}^n$ is defined by $d_{p,min}/\lambda^n$, where $\lambda$ is the minimum norm of $\Lambda$.
(2) The normalized minimum product distance $d_{p,nor}$ of a complete lattice $\Lambda\subset\mathbb{R}^n$ is defined by $d_{p,min}/vol(\Lambda)$.
\end{defin}

It is usual to consider lattices having same volume (or same determinant), but since in our case the comparison is between two lattices having different volumes, we need use some of these two definitions. The relative minimum product distance is used in \cite{grasi2} also to compare $\mathbb{Z}^n$ and $D_n$. However, in other contexts it seems better use the normalized minimum product distance, like when someone needs a lattice having a greater number of points inside a same convex polytope. Intuitively, suppose that in a lattice we put tiny cubes centered in the lattice points. Refining this lattice, the sum of the volume of cubes inside the fundamental region of the lattice approximates of its volume, although volume of the refined lattices became smaller. This procedure remember the basic principle of the Riemann integral. Therefore, higher volume, lower number of points inside a fixed convex polytope region.

\begin{ex} Consider in $\mathbb{R}^2$ the lattice $\Lambda_1$ generated by $\{(\sqrt{3}/2,-1/2), (1/2,\sqrt{3}/2)\}$ and its sublattice $\Lambda_2$ generated by $\{(\sqrt{3}/2,-1/2), (1,\sqrt{3})\}$. Note that $\Lambda_1$ has more points than $\Lambda_2$ in the fundamental region of $\Lambda_2$. So, if the parameter for choosing the lattice is the largest number of points in a same region, is better choose $\Lambda_1$. To force $\Lambda_2$ to be better than $\Lambda_1$ in this analysis we consider the scaled version $(1/vol(\Lambda_2))\Lambda_2=(1/2)\Lambda_2$, that has more points than $\Lambda_1$ in the fundamental region of $\Lambda_2$. As $vol(\Lambda_2)=2\neq 1= vol(\Lambda_1)$ then $d_{p,nor}(\Lambda_2)\neq d_{p,nor}(\Lambda_1)$ if we rescale $\Lambda_1$ and $\Lambda_2$ to have same minimum product distance. However, note that $\Lambda_1$ and $\Lambda_2$ have the same minimum norm $\lambda$, implying that $d_{p,rel}(\Lambda_1)=d_{p,rel}(\Lambda_2)$ if we force them to have same minimum product distance. Therefore, $d_{p,nor}$ is more effective to compare $\Lambda_1$ and $\Lambda_2$ than $d_{p,rel}$ in this case.
\end{ex}
 
Firstly, we will use the relative minimum product distance to compare the constructed lattices $\mathbb{Z}^n$ and $D_n$. In this case we can make a parallel with the results in \cite{grasi2}. After, in the end of this section, we compare these two lattices using normalized minimum product distance and present a table comparing them for $3\leq n \leq 10$.

The minimum distance of $\mathbb{Z}^n$ is $1$. However, $d_{p,rel}$ of $\mathbb{Z}^n$ coincides with $d_{p,min}$ in each construction that we got calculate the minimum product distance. In turn, the minimum distance of $D_n$ is $\sqrt{2}$, which implies that $d_{p,rel}=2^{-n/2}d_{p,min}$ in each construction that we got calculate the minimum product distance. So, in these cases, if $n>1$ is a odd number,
$$\frac{\sqrt[n]{d_{p,rel}(\mathbb{Z}^n)}}{\sqrt[n]{d_{p,rel}(D_n)}}=\sqrt{2}$$
if $n>1$ is a power of two,
$$\frac{\sqrt[n]{d_{p,rel}(\mathbb{Z}^n)}}{\sqrt[n]{d_{p,rel}(D_n)}}=2^{\frac{1}{2}-\frac{1}{n}}$$
and if $n$ is a product of a odd number ($>1$) by a power of two ($>1$) then
$$\frac{\sqrt[n]{d_{p,rel}(\mathbb{Z}^n)}}{\sqrt[n]{d_{p,rel}(D_n)}}=\frac{\sqrt{2}}{2}$$
while
\begin{equation}\label{eq9}
\lim_{\eta\longrightarrow \infty} \frac{\sqrt[n]{\delta(\mathbb{Z}^n)}}{\sqrt[n]{\delta(D_n)}}=0.
\end{equation}
Therefore, it is better use $D_n$ than $\mathbb{Z}^n$ when someone needs a lattice having good performance simultaneously for Rayleigh fading channel and for gaussian channel in high dimensions. Notably, observe that if $n$ is a product of a odd number ($>1$) by a power of two ($>1$) then $d_{p,rel}$ and $\delta$ of $D_n$ are better than those of $\mathbb{Z}^n$.

Throughout this work we were able to calculate the minimum product distance of the lattices $\mathbb{Z}^l$ and $D_l$, with odd $l$, only when certain conditions were valid (see Corollary \ref{corol1} and Corollary \ref{teorema3}). In $\mathbb{Z}^l$ we needed $u=\sigma(x)/x$ to be an algebraic integer, while in $D_l$ we needed $1+u$ or $1-u$ to be unit. The table \ref{tab1} compares the results of the $l$-th root of the relative minimum product distance of $\mathbb{Z}^l$ with that of $D_l$ for $l$ between $3$ and $15$ using the least prime number $p$ such that $p\equiv 1~(mod~l)$, except for $l=13$ (due to remarks \ref{obs3} and \ref{obs2}). In the same table we can compare the center density of the lattices in each dimension. The column "Unit" presents which of $1+u$ or $1-u$ is unit.

\begin{table}
\begin{center}
\caption{Comparison between relative minimum product distance of $\mathbb{Z}^n$ and $D_n$ and between their center density in some odd dimensions $n$}\label{tab1}
\begin{tabular}{|c|c|c|c|c|c|c|c|c|}
\hline
$l$ & $p$ & $r$ & $u\in \mathbb{Z}[\zeta_p]$? & Unit & $\sqrt[l]{d_{p,rel}(\mathbb{Z}^l)}$ & $\sqrt[l]{d_{p,rel}(D_l)}$ & $\delta(\mathbb{Z}^l)$ & $\delta(D_l)$ \\ \hline
$3$ & $7$ & $3$ & Yes (Theorem \ref{teorema4}) & $1\pm u$ & $0.5227...$ & $0.3696...$ & $0.125$ & $0.176776...$\\ \hline
$5$ & $11$ & $2$ & Yes (Theorem \ref{teorema4}) & $1\pm u$ & $0.3832...$ & $0.2709...$& $0.031250...$ & $0.088388...$ \\ \hline
$7$ & $29$ & $2$ & Yes & $1-u$ & $0.2361...$ & $0.1670...$& $0.007812...$ & $0.044194...$ \\ \hline
$9$ & $19$ & $2$ & Yes (Theorem \ref{teorema4}) &$1\pm u$ & $0.2701...$ & $0.1910...$& $0,2701...$ & $0,1910...$\\ \hline
$11$ & $23$ & $5$ & Yes (Theorem \ref{teorema4}) & $1\pm u$ & $0.2404...$ & $0.1700...$& $0,000488...$ & $0.011048...$\\ \hline
$15$ & $31$ & $3$ & Yes (Theorem \ref{teorema4}) & $1\pm u$ & $0.2013...$ & $0.1424...$& $0,000030...$ & $0.002762...$\\ \hline

\end{tabular}
\end{center}
\end{table}

The cases $l=5$, $l=9$ and $l=11$ were stutied in \cite{grasi2} too. Note that $\mathbb{Z}^l$ and $D_l$ obtained here and in \cite{grasi2} have the same relative minimum product distance if $1+u$ or $1-u$ is as unit in $\mathcal{O}_\mathbb{K}$ (see table 3 of \cite{grasi2}). However, in \cite{grasi2} only lattices satisfying the condition $2l=p-1$ were constructed. For example, in \cite{grasi2} it was not possible to construct $\mathbb{Z}_7$ and $D_7$ neither calculate their minimum product distance. Here this is possible.

In relation to general constructions of $\mathbb{Z}^n$ and $D_n$ for any integer number $n>1$, again we can calculate the minimum product distance if some hypothesis are valid (see theorems \ref{teorema8} and \ref{teorema9}). In the table \ref{tab2}, we compare the relative minimum product distance between $\mathbb{Z}^n$ and $D_n$ in each even dimension $n$ between $4$ and $14$. In particular, we can highlight the observed value of the $n$-th root of the relative minimum product distance of $D_6$. In \cite{grasi2} it was constructed a lattice $D_6$ having $\sqrt[6]{d_{p,rel}(D_6)} \simeq 0.24285$, while using the construction made here this value gets to $\simeq 0.4395$.

\begin{table}
\begin{center}
	\caption{Comparison between relative minimum product distance of $\mathbb{Z}^n$ and $D_n$ and between their center density in some even dimensions $n$}\label{tab2}
\begin{tabular}{|c|c|c|c|c|c|c|c|c|}
	\hline
	$n$ & $l$ ($p$) & $m$ & $k=2^{m-2}$ & $\sqrt[\eta]{d_{p,rel}(\mathbb{Z}^n)}$ & $\sqrt[\eta]{d_{p,rel}(D_n)}$ & $\delta(\mathbb{Z}^n)$ & $\delta(D_n)$ \\ \hline
	$4$ & $1$ & $4$ & $k=2^2$ & $0.3855...$ & $0.3242...$ & $0.0625...$ & $0.125$ \\ \hline
	$6$ & $3$ ($p=7$) & $3$ & $k=2^1$ & $0.3108...$ & $0.4395...$ & $0.015625$ & $0.0625$ \\ \hline
	$8$ & $1$ & $5$ & $k=2^3$ & $0.2610...$ & $0.2013...$ & $0.003906...$ & $0.03125$ \\ \hline
	$10$ & $5$ ($p=11$) & $3$ & $k=2^1$ & $0.2278...$ &$0.3222...$ &$0.000976...$ &$0.015625$ \\ \hline
	$12$ & $3$ ($p=7$) & $4$ & $k=2^2$ & $0.2015...$ &$0.2850...$ &$0.000244...$ &$0.0078125$ \\ \hline
	$14$ & $7$ ($p=29$) & $3$ & $k=2^1$ & $0.1404...$ & $0.1986...$ & $0.000061...$ & $0.00390625$\\ \hline
	\end{tabular}
\end{center}
\end{table}

Finally, let's compare $\mathbb{Z}^n$ and $D_n$ using the normalized minimum product distance. Since the volume of the constructed $\mathbb{Z}^n$ is $1$ and the volume of its sublattice $D_n$ is $2$ then $d_{p,nor}(\mathbb{Z}^n)=d_{p,min}(D_n)$ and $d_{p,nor}(D_n)=d_{p,min}(D_n)/2$. So, for the constructions that we can calculate the minimum product distance of $\mathbb{Z}^n$ and $D_n$ in this work, if $n>1$ is an odd number,
$$\frac{\sqrt[\eta]{d_{p,nor}(\mathbb{Z}^n)}}{\sqrt[n]{d_{p,nor}(D_n)}}=\sqrt[n]{2} \xrightarrow{n \longrightarrow \infty} 1$$
if $n>1$ is a power of two,
$$\frac{\sqrt[n]{d_{p,nor}(\mathbb{Z}^n)}}{\sqrt[n]{d_{p,nor}(D_n)}}=1$$
and if $n$ is a product of an odd number ($>1$) by a power of two ($>1$) then
$$\frac{\sqrt[n]{d_{p,nor}(\mathbb{Z}^n)}}{\sqrt[n]{d_{p,nor}(D_n)}}=2^{\frac{1}{n}-1} \xrightarrow{n \longrightarrow \infty} 1/2.$$
Since the quotient in \ref{eq9} is valid here too, we can conclude that is better use $D_n$ than $\mathbb{Z}_n$ for both Rayleigh fadind channel and gaussian channel for any $n>1$. In fact, even in the case in which $n>1$ is a odd number, $\mathbb{Z}^n$ loses its vantage as $n$ grows.

\begin{table}
	\begin{center}
		\caption{Comparison between normalized minimum product distance of $\mathbb{Z}^n$ and $D_n$ and between their center density in some even dimensions $n$}\label{tab3}
		\begin{tabular}{|c|c|c|c|c|c|c|c|c|}
			\hline
			$n$ & $l$ ($p$) & $m$ & $k=2^{m-2}$ & $\sqrt[n]{d_{p,nor}(\mathbb{Z}^n)}$ & $\sqrt[n]{d_{p,nor}(D_n)}$ & $\delta(\mathbb{Z}^n)$ & $\delta(D_n)$ \\ \hline
			$3$ & $3$ ($p=7$) & $2$ & $k=1$ & $0.1428...$ & $0.1133...$ & $0.0625...$ & $0.125$ \\ \hline
			$4$ & $1$ & $4$ & $k=2^2$ & $0.3855...$ & $0.3855...$ & $0.0625...$ & $0.125$ \\ \hline
			$5$ & $5$ ($p=11$) & $2$ & $k=1$ & $0.0082...$ & $0.0071...$ & $0.0625...$ & $0.125$ \\ \hline
			$6$ & $3$ ($p=7$) & $3$ & $k=2^1$ & $0.3108...$ & $0.5538...$ & $0.015625$ & $0.0625$ \\ \hline
			$7$ & $7$ ($p=29$) & $2$ & $k=1$ & $0.00004...$ & $0.00003...$ & $0.0625...$ & $0.125$ \\ \hline
			$8$ & $1$ & $5$ & $k=2^3$ & $0.2610...$ & $0.2610...$ & $0.003906...$ & $0.03125$ \\ \hline
			$9$ & $9$ ($p=19$) & $2$ & $k=1$ & $0.000007...$ & $0.000007...$ & $0.0625...$ & $0.125$ \\ \hline
			$10$ & $5$ ($p=11$) & $3$ & $k=2^1$ & $0.2278...$ &$0.4252...$ &$0.000976...$ &$0.015625$ \\ \hline
		\end{tabular}
	\end{center}
\end{table}

In the table \ref{tab3} we can observe that the values of $\sqrt[n]{d_{p,nor}}$ of $\mathbb{Z}^n$ and $D_n$ are not good when $n$ is a odd number, but it is good when $n$ is the product of a odd number ($>1$) by a power of two ($>1$).

Therefore, we can conclude that $D_n$ has interesting advantages for practical use when compared with $\mathbb{Z}^n$ both for Rayleigh fading channel and for gaussian channel.

\section{Acknowledgment}

The authors would like to thank Dr. Trajano P. N. Neto and Dr. Grasiele C. Jorge for suggestions and discussions about this work. The authors also thank Dr. Frederique Oggier for the readiness to answer some of our doubts via e-mail.

\end{document}